\documentclass[12pt]{article}
\title{Smoothness of solutions to 
	\\
	the initial-boundary value problem for 
	\\
	the telegraph equation on the half-line with 
	\\
	a locally summable potential}
\author{S.\,A.\,Simonov
	\thanks{St. Petersburg Department of Steklov Mathematical Institute of Russian Academy of Sciences, Fontanka 27, St. Petersburg, Russia, 191023; Alferov Academic University of the Russian Academy of Sciences, Khlopina 8/3, St. Petersburg, Russia; sergey.a.simonov@gmail.com}}
\date{}

\usepackage{enumitem}
\usepackage{amssymb,amsmath,amsthm}
\usepackage{mathrsfs}
\usepackage[colorlinks=true]{hyperref}

\newtheorem{Lemma}{Lemma}
\newtheorem{Theorem}{Theorem}

\newtheorem{Remark}{Remark}
\newtheorem{Corollary}{Corollary}

\def\matr{\mathbb M^n_{\mathbb C}}
\def\vt{\widetilde v}
\def\wt{\widetilde w}

\def\mF{\mathscr F}

\def\mH{\mathscr H}
\def\mK{\mathscr K}

\def\mU{\mathscr U}

\def\Dom{{\rm Dom\,}}
\def\Ker{{\rm Ker\,}}

\def\supp{{\rm supp\,}}
\def\dx{\frac{d^2}{dx^2}}

\def\gs{\geqslant}

\def\rmv{_{\rm v}}
\def\bb{\mathbb}
\def\bbc{\mathbb C^n}
\def\ot{[0,T]}
\def\otbbc{\ot;\bbc}
\def\otk{\ot;\mK}

\begin{document}
\maketitle

\begin{abstract}
We study solutions to the system
\begin{align*}
	&u_{tt}-u_{xx}+q(x)u=0, & x>0,\ t>0,\\
	&u|_{t=0}=u_t|_{t=0}=0,  & x\gs0,\\
	&u|_{x=0}=g(t), & t\gs0,
\end{align*}
with a locally summable Hermitian matrix-valued potential $q$ and a $\mathcal C^{\infty}$-smooth $\bb C^n$-valued boundary control $g$ vanishing near the origin. We prove that the solution $u^{g}(\cdot,T)$ is a function from $\mathcal W^2_1(\otbbc)$ and that the control operator $W^T:g\mapsto u^{g}(\cdot,T)$ is an isomorphism in $\mathcal L_2(\otbbc)$, and, in the case $q\in \mathcal L_2(\otbbc)$, also an isomorphism in $\mathcal H^2(\otbbc)$.
\end{abstract}

\section{Introduction}

Let $L_0$ be a symmetric positive definite completely non-self-adjoint operator in a Hilbert space $\mH$. Consider a dynamical system with boundary control, $\alpha$,
\begin{align}
\label{alpha1}	&u(t)\in\Dom L_0^*, & t\gs0,\\
\label{}	&u''(t)+L_0^*u(t) = 0,  & t>0,\\
\label{}	&u(0)=u'(0)=0, & \\
\label{alpha4}	&\Gamma_1 u(t) = f(t), & t\gs0.
\end{align}
Here $\Gamma_1$ is the first of the two boundary operators \cite{MMM} of the so-called Vishik's boundary triple $(\mK;\Gamma_1,\Gamma_2)$, $\mK:=\Ker L_0^*$, $\dim\mK=n_{\pm}(L_0)$ (the defect indices), $\Gamma_1,\Gamma_2:\Dom L_0^*\to\mK$, $\Gamma_1:=L^{-1}L_0^*-I$, $\Gamma_2:=P_{\mK}L_0^*$ ($P_{\mK}$ is the projection on $\mK$ in $\mH$ and $L$ is the Friedrichs extension of $L_0$). The construction of this boundary triple is based on the Vishik's decomposition \cite{Vishik} $\Dom L_0^*=\Dom L_0\dot{+}\mK\dot{+}L^{-1}\mK$. The {\it boundary control} $f$ is a function from $\mathcal C^{\infty}([0,\infty);\mK)$ with $\supp f\subset(0,\infty)$. 
In the boundary control approach to inverse problems (the BC-method, \cite{1988,1997,
	2017,BD_DSBC}) one usually considers a family (a {\it nest}) of subsystems $\alpha^T$ of the system $\alpha$ on finite intervals $t\in[0,T]$. 
For each $T>0$ one is in particular interested in the {\it ``smooth waves''} which are solutions of $\alpha$ at the time $T$, $u^f(T)\in\mH$. 
Smooth waves form the linear manifold $\dot\mU^T$ which is called the {\it reachable set}, and its closure is called the {\it reachable subspace} $\mU^T$ of the system $\alpha^T$. 
Smooth controls form the class $\dot\mF^T:=\{f\in \mathcal C^{\infty}([0,T];\mK)\,|\,\supp f\subset(0,T]\}$, which is dense in the {\it outer space} $\mF^T:=\mathcal L_2([0,T];\mK)$ of $\alpha^T$. 
The space $\mH$ plays the role of the {\it inner space} of $\alpha^T$, and the correspondence between controls and solutions is established by the {\it control operator} $W^T:\mF^T\to\mH$ acting by the rule $f\mapsto u^f(T)$ on $\Dom W^T=\dot\mF^T$. 
Properties of the control operator are determined by the operator $L_0$ and play an important role in the BC-method. 
For example, for the problem of electromagnetic tomography $W^T$ is unbounded \cite{2017}, for the problems of acoustic tomography $W^T$ is bounded, but can be not invertible for $T$ large enough \cite{2017}, and for the Schr\"odinger operator on the half-line $W^T$ is an {\it isomorphism} from $\mF^T$ to $\mU^T$ (i.\,e., a bounded operator which has a bounded inverse from $\mU^T$ to $\mF^T$) \cite{BSim_A&A_2017}. 
The nest of projections on reachable subspaces $\{P_{\mU^T}\}_{T\in[0,\infty]}$ serves as a resolution of the identity for the {\it eikonal operator} $E=\int_{[0,\infty)}TdP_{\mU^T}$ in $\mH$. 
Passing to the spectral representation of this operator, we obtain the {\it model space} of functions where the operator $L_0$ has its own representation $\tilde L_0$, which is considered as a {\it functional model} of $L_0$ \cite{BSim_A&A_2024,BSim_ZNSP_2023}. 
From many possible spectral representations the one based on the {\it triangular factorization} \cite{GK} of the {\it connecting operator} $C^T:=(W^T)^*W^T$ with respect to the nest of delayed control subspaces $\{\mF^T_s\}_{s\in[0,T]}$, $\mF^T_s:=\mathcal L_2([T-s,T];\mK)$, plays a distinguished role, because it leads to a {\it local} functional model of $L_0$, which provides a way to solve inverse problems. 
It is based on finding the {\it diagonal} $D_{W^T}$ of the control operator \cite{BPush} with respect to the nest $\{\mF^T_s\}_{s\in[0,T]}$. 
This is why the properties of the control operator are very important in the BC-method.

In the present paper we consider the minimal matrix Schr\"odinger operator $L_0=-\dx+q(x)$ in the Hilbert space $\mH=\mathcal L_2([0,\infty);\mathbb C^n)$ with a locally summable Hermitian matrix-valued potential $q$ and study properties of the corresponding smooth waves and the control operator. For such $L_0$ we show that the system $\alpha^T$ takes a form of the system $\alpha^T\rmv$:
\begin{align}
\label{alpha_v1}&u(\cdot,t)\in\Dom L_0^*,&&t\in\ot,\\
\label{alpha_v2}&\ddot u(x,t)-u_{xx}(x,t)+q(x)u(x,t) = 0,  && x\in(0,\infty),t\in(0,T),\\
\label{alpha_v3}&u(x,0)=\dot u(x,0)=0, && x\in[0,\infty),\\
\label{alpha_v4}&u(0,t) = f\rmv(t), && t\in[0,T],
\end{align}
where $f\rmv\in\mF^T\rmv:=\mathcal L_2([0,T];\mathbb C^n)$ is simply related to the control $f$ so that its solution $u^{f\rmv}_{\alpha}(x,t)$ coincides with $u^f(x,t)$ (see \eqref{Lambda} below) and the differentiation has to be understood in the following sense: $\dot u$ as a derivative of an $\mH$-valued function in $t$ and $u_x$ as a partial derivative of a vector-valued function of two variables $x$ and $t$ with respec to to $x$. The system $\alpha^T\rmv$ is almost the same as the initial-boundary value problem $\beta^T\rmv$ for the telegraph equation
\begin{align}
\label{beta1}&u_{tt}(x,t)-u_{xx}(x,t)+q(x)u(x,t) = 0,  && x\in(0,\infty),t\in(0,T),\\
\label{}	&u(x,0)=u_t(x,0)=0, && x\in[0,\infty),\\
\label{beta3}	&u(0,t) = f\rmv(t), && t\in[0,T],
\end{align}
where $u_t$ denotes a partial derivative with respect to $t$. If the potential $q$ and the control $f\rmv$ are both smooth and the control vanishes near $0$, it is well known that the system \eqref{beta1}--\eqref{beta3} has the classical solution
\begin{equation}\label{u^f repres}
	u^{f\rmv}_{\beta}(x,t):=f\rmv(t-x)+\int_x^tw(x,s)f\rmv(t-s)ds,
\end{equation}
where $w$ is a smooth kernel. One can see that this solution also solves the system $\alpha^T\rmv$ and thus $W^Tf=u^f(T)=u^{f\rmv}_{\alpha}(T)=u^{f\rmv}_{\beta}(T)$ holds. This defines the corresponding control operator $W^T\rmv:f\rmv\mapsto u^{f\rmv}_{\alpha}(T)$, which is an isomorphism in $\mathcal L_2(\otbbc)$ (this is clear from \eqref{u^f repres}). Therefore $W^T$ (its closure) is an isomorphism from $\mF^T=\mathcal L_2(\otk)$ to $\mU^T=\mathcal L_2(\otbbc)$. For smooth contols $f(t)$ the solution $u^f(\cdot,T)$ is smooth, and if one considers the restiction of $W^T\rmv$ to $\mathcal H^2(\otbbc)$, this operator turns out to be an isomorphism in this space (in the norm of this space).

Note that in the general situation one has the abstract representation 
\begin{equation}\label{abstract representation for solution}
	u^f(t)=-f(t)+\int_0^tL^{-\frac12}\sin(L^{\frac12}(t-s))f''(s)ds
\end{equation}
for the solution of the system $\alpha^T$, but it does not help with establishing of the properties of the smooth waves discussed above. However, it helps to see that $u^f(t)$ is unique and smooth in $t$ as an $\mH$-valued function.

In the present work we consider potentials $q$ from the class $\mathcal L_{1,\rm loc}([0,\infty);\matr)$, where $\matr$ denotes square matrices of size $n$ with complex entries. 
In Theorem \ref{theorem solution} we show that, although the kernel $w(x,t)$ is only $\mathcal W^1_1$-smooth in $x$, the solution $u^{f\rmv}_{\alpha}(x,t)$ is $\mathcal W^2_1$-smooth in $x$ (and $\mathcal C^{\infty}$-smooth in $t$). 
The operator $W^T$ is closely related to the operator $W^T\rmv$, which is an isomorphism in $\mathcal L_2([0,T];\mathbb C^n)$. 
In Theorem \ref{theorem isomorphism} we show that if $q\in \mathcal L_2([0,T];\matr)$, then $W^T\rmv$ can be considered in the Sobolev space $\mathcal H^2([0,T];\mathbb C^n)$, is a bounded operator there and, moreover, such a restriction is an isomorphism in $\mathcal H^2([0,T];\mathbb C^n)$. 
These results will be used in the forthcoming paper \cite{forthcoming} on characterization of operators unitarily equivalent to matrix Schr\"odinger operators.

The paper is organized as follows. In Section \ref{section 2} we introduce the operator $L_0$ and derive the equivalent system $\alpha^T\rmv$ from the system $\alpha^T$. In Section \ref{section 3} we consider the initial-boundary value problem $\beta^T\rmv$ and prove a formula for its solution by analyzing a corresponding Goursat problem. In Section \ref{section 4} we relate the systems $\alpha^T\rmv$ and $\beta^T\rmv$ showing that solutions of the latter satisfy the former. In Section \ref{section 5} we study the ``smoothness'' properties of the operator $W^T\rmv$, i.\,e., its boundedness and invertibility in the Sobolev norm.

\section{The minimal Schr\"odinger operator}\label{section 2}

Let $q$ be a Hermitian matrix-valued function from the class $\mathcal L_{1,\rm loc}([0,\infty);\matr)$, and consider the minimal and the maximal matrix Schr\"odinger operators $L_{\rm min}$ and $L_{\rm max}$ in the Hilbert space $\mH=\mathcal L_2([0,\infty);\bbc)$ on the domains
\begin{align*}
&\Dom L_{\rm max}\\
&=\{y\in \mathcal L_2([0,\infty);\mathbb C^n)\cap \mathcal W^2_{1,\rm loc}([0,\infty);\mathbb C^n)\,|\,-y''+qy\in \mathcal L_2([0,\infty);\mathbb C^n)\},\\
&\Dom L_{\rm min}=\{y\in\Dom L_{\rm max}\,|\,\supp y\subset(0,\infty)\text{ is compact}\}.
\end{align*}
Let $L_0$ be the closure of $L_{\rm min}$, then $L_0^*=L_{\rm max}$.
Suppose that the potential $q$ is such that the operator $L_0$ is positive definite. Then by the Povzner--Wienholtz theorem \cite{Clark-Gesztesy} its defect indices are $n_{\pm}(L_0)=n$ and the domains of $L_0$ and its Friedrichs extension $L$ are as follows:
\begin{align*}
&\Dom L=\{y\in\Dom L_0^*\,|\,y(0)=0\},\\
&\Dom L_0=\{y\in\Dom L_0^*\,|\,y(0)=y'(0)=0\}.
\end{align*}
The kernel of $L_0^*$ consists of solutions to the eigenfunction equation for the zero spectral parameter,
\begin{equation*}
\mK=\Ker L_0^*=\{y\in\Dom L_0^*\,|\,-y''+qy=0\},\quad\dim\mK=n.
\end{equation*}
The Vishisk's decomposition $\Dom L_0^*=\Dom L_0\dot+L^{-1}\mK\dot+\mK=\Dom L\dot+\mK$ enables one to define the boundary triple $(\mK;\Gamma_1,\Gamma_2)$ with the boundary operators $\Gamma_1=L^{-1}L_0^*-I$, $\Gamma_2=P_{\mK}L_0^*$, which for $y=y_0+L^{-1}g+h\in\Dom L_0^*$, $y_0\in\Dom L_0$, $h,g\in\mK$, act as $\Gamma_1y=-h$, $\Gamma_2y=g$. Using $\Gamma_1$ we consider the dynamical system with boundary control $\alpha^T$, \eqref{alpha1}--\eqref{alpha4}, where $f\in\dot\mF^T$, i.\,e., $f\in \mathcal C^{\infty}([0,T];\mK)$ and $\supp f\subset(0,T]$. 

Let $K(x)$ be the matrix-valued solution to the equation $-y''+qy=0$ such that its columns $k_1$, ..., $k_n$ form a base in $\mK$ and $K(0)=I_{\matr}$. Such a solution exists, because for any base $k_1,...,k_n\in\mK$ one has $\det K(0)\neq0$: otherwise there exists a linear combination $k(x)=\sum_{i=1}^nc_ik_i(x)$ such that $k(0)=0$ and thus $k\in\Dom L\cap\mK$, but this intersection is trivial by the Vishik's decomposition.

To find $\Gamma_1$ and $\Gamma_2$ one needs to find $g$ and $h$ in the decomposition $y=y_0+L^{-1}g+h$. Since $h,g\in\mK$, let $g(x)=K(x)c$, $h(x)=K(x)d$ with $c,d\in\bbc$. Denote $K_1(x)=(L^{-1}K)(x)$, where the operator is applied to the matrix $K(x)$ column-wise. Since $y_0\in\Dom L_0$, one has $y_0(0)=0$ and $y_0'(0)=0$, hence
\begin{equation*}
y(0)=h(0)=K(0)d=d,\quad y'(0)=K_1'(0)c+K'(0)d,
\end{equation*}
and
\begin{align*}
&(\Gamma_1y)(x)=-h(x)=-K(x)d=-K(x)y(0),\\
&(\Gamma_2y)(x)=g(x)=K(x)c=K(x)(K_1'(0))^{-1}(y'(0)-K'(0)y(0)).
\end{align*}
Here the matrix $K_1'(0)$ is invertible, because if $\det K_1'(0)=0$, then there exists a non-trivial linear combination $k(x)=\sum_{i=1}^nc_ik_i(x)$ such that $(L^{-1}k)'(0)=0$, and together with the fact that $(L^{-1}k)(0)=0$ this means that $L^{-1}k\in\Dom L_0\cap(L^{-1}\mK)=\{0\}$, but this intersection is again trivial by the Vishik's decomposition.

Define the isomorphism $\lambda:\bbc\to\mK$, $(\lambda v)(x)=-K(x)v$, and the corresponding isomorphism $\Lambda^T:\mathcal L_2(\otbbc)\to \mathcal L_2(\otk)$,
\begin{equation*}
(\Lambda^T f\rmv)(t):=\lambda(f\rmv(t)),\quad t\in\ot.
\end{equation*}
Parametrizing 
\begin{equation}\label{Lambda}
f=\Lambda^T f\rmv,
\end{equation}
we rewrite the system $\alpha^T$, \eqref{alpha1}--\eqref{alpha4}, in the equivalent form $\alpha^T\rmv$, \eqref{alpha_v1}--\eqref{alpha_v4}, and the control operators $W^T$ and $W^T\rmv$ are related as
\begin{equation}\label{relation btw W}
	W^T\rmv=W^T\Lambda^T.
\end{equation}

\section{The initial-boundary value problem $\beta^T\rmv$}\label{section 3}

In this section we establish properties of the solution to the system $\beta^T\rmv$, \eqref{beta1}--\eqref{beta3}. We show that there exists a continuous matrix-valued function $w(x,t)$, defined for $0\leqslant x\leqslant t<\infty$, such that the function
\begin{equation}\label{u hat f}
u^{f\rmv}_{\beta}(x,t)=f\rmv(t-x)+\int_x^tw(x,s)f\rmv(t-s)ds,\quad x\in[0,\infty),\ t\in[0,T],
\end{equation}
under the agreement that $f\rmv$ is continued to the negative half-line by zero, is the solution of the system $\beta^T\rmv$. 

The kernel $w$ is a generalized solution of the following Goursat problem:
\begin{align}
\label{w1}&w_{tt}(x,t)-w_{xx}(x,t)+q(x)w(x,t)=0, && t>0,\  x\in(0,t),\\
&w(0,t)=0,  && t\geqslant0,\\
\label{w3}&w(x,x)=-\frac12\int_0^xq(s)ds, && x\geqslant0,
\end{align}
in the sense that it is the solution of the corresponding integral equation. In what follows we derive this equation in the same way as it was done in \cite{Avdonin-Mikhailov-2010} and show that it has the unique continuous solution. Then we prove that $u^{f\rmv}_{\beta}$ defined by \eqref{u hat f} satisfies \eqref{beta1}--\eqref{beta3}.

To reduce the Goursat problem to an integral equation let us change the variables,
$$
\xi=t-x,\quad \eta=t+x,\quad v(\xi,\eta):=w\left(\frac{\eta-\xi}{2},\frac{\eta+\xi}{2}\right).
$$
Then \eqref{w1}--\eqref{w3} becomes
\begin{align}
\label{v1}&v_{\xi\eta}(\xi,\eta)+q\left(\frac{\eta-\xi}{2}\right)\frac{v(\xi,\eta)}4=0, && \eta>0,\ \xi\in(0,\eta),\\
\label{v2}&v(\xi,\xi)=0,  && \xi\geqslant0,\\
\label{v3}&v(0,\eta)=-\frac12\int_0^{\frac{\eta}{2}}q(s)ds, && \eta\geqslant0.
\end{align}
Formally solving the inhomogeneous differential equation $v_{\xi\eta}(\xi,\eta)=g(\xi,\eta):=-q\left(\frac{\eta-\xi}{2}\right)\frac{v(\xi,\eta)}4$ gives $v(\xi,\eta)=\int_0^{\xi}d\xi_1\int_0^{\eta}d\eta_1g(\xi_1,\eta_1)+v_1(\xi)+v_2(\eta)$ with some functions $v_1$ and $v_2$ that can be determined from \eqref{v2} and \eqref{v3}. This leads to the integral equation
\begin{equation}\label{v int eq}
v(\xi,\eta)=-\frac12\int_{\frac{\xi}{2}}^{\frac{\eta}{2}}q(s)ds-\frac14\int_0^{\xi}d\xi_1\int_{\xi}^{\eta}d\eta_1q\left(\frac{\eta_1-\xi_1}{2}\right)v(\xi_1,\eta_1).
\end{equation}
Denote
\begin{equation*}
v_0(\xi,\eta):=-\frac12\int_{\frac{\xi}{2}}^{\frac{\eta}{2}}q(s)ds,
\quad
\vt:=Vv,
\end{equation*}
where $V$ is the operator acting as
\begin{equation*}
(Vv)(\xi,\eta):=-\frac14\int_0^{\xi}d\xi_1\int_{\xi}^{\eta}d\eta_1q\left(\frac{\eta_1-\xi_1}{2}\right)v(\xi_1,\eta_1).
\end{equation*}
Then the integral equation \eqref{v int eq} can be written as $v=v_0+Vv$ and its solution is $v=(I-V)^{-1}v_0=\sum_{n=0}^{\infty}V^nv_0$, if one shows that the series converge in a suitable sense. Let us fix $T>0$ and consider it in the space $\mathcal C(\{(\xi,\eta)\,|\,\eta\in[0,T],\xi\in[0,\eta]\})$. One has
\begin{equation*}
\|v_0(\xi,\eta)\|_{\matr}\leqslant\frac12\int_0^{\frac{\eta}{2}}\|q(s)\|_{\matr}ds=:S(\eta),\quad \forall\xi\geqslant0,
\end{equation*}
\begin{multline*}
\|(Vv_0)(\xi,\eta)\|_{\matr}
\leqslant
\frac14\int_0^{\xi}d\xi_1\int_{\xi}^{\eta}d\eta_1S(\eta_1)\left\|q\left(\frac{\eta_1-\xi_1}2\right)\right\|_{\matr}
\\
=
S(\eta)\int_0^{\xi}d\xi_1\int_{\xi}^{\eta}d\eta_1S'(\eta_1-\xi_1)\leqslant S^2(\eta)\xi,
\end{multline*}
analagously
\begin{multline*}
\|(V^2v_0)(\xi,\eta)\|_{\matr}
\leqslant
\frac14\int_0^{\xi}d\xi_1\int_{\xi}^{\eta}d\eta_1S^2(\eta_1)\xi_1\left\|q\left(\frac{\eta_1-\xi_1}2\right)\right\|_{\matr}
\\
=
S^2(\eta)\int_0^{\xi}d\xi_1\xi_1\int_{\xi}^{\eta}d\eta_1S'(\eta_1-\xi_1)\leqslant \frac{S^3(\eta)\xi^2}2,
\end{multline*}
and, by induction,
$\|(V^nv_0)(\xi,\eta)\|_{\matr}\leqslant\frac{S^{n+1}(\eta)\xi^n}{n!}$.
Therefore the series converge and the resulting solution $v$ belongs to  $\mathcal C(\{(\xi,\eta)\,|\,\eta\in[0,T],\xi\in[0,\eta]\})$ for every $T>0$ and satisfies the estimate
\begin{equation*}
\|v(\xi,\eta)\|_{\matr}\leqslant S(\eta)e^{\xi S(\eta)}.
\end{equation*}
Looking at the equality \eqref{v int eq}, we see that $v$ is differentiable and that
\begin{equation*}
v_{\xi}(\xi,\eta)=\frac14q\left(\frac{\xi}2\right)
-\frac14\int_{\xi}^{\eta}q\left(\frac{\eta_1-\xi}2\right)v(\xi,\eta_1)d\eta_1
+\frac14\int_0^{\xi}q\left(\frac{\xi-\xi_1}2\right)v(\xi_1,\xi)d\xi_1,
\end{equation*}
\begin{equation*}
v_{\eta}(\xi,\eta)=-\frac14q\left(\frac{\eta}2\right)
-\frac14\int_0^{\xi}q\left(\frac{\eta-\xi_1}2\right)v(\xi_1,\eta)d\xi_1.
\end{equation*}
Moreover,
\begin{equation*}
v_{\xi\eta}(\xi,\eta)=-\frac14q\left(\frac{\eta-\xi}{2}\right)v(\xi,\eta)
\end{equation*}
holds for fixed $\xi$ and a.\,e. $\eta$ and for fixed $\eta$ and a.\,e. $\xi$, which gives \eqref{v1}. Conditions \eqref{v2} and \eqref{v3} are obviously satisfied, so $v$ solves the problem \eqref{v1}--\eqref{v3}. Note that the derivatives $v_{\xi\xi}$ and $v_{\eta\eta}$ do not exist.

Solution $v$ is a sum of the explicit part $v_0$ which is only differentiable and the part $\vt$ which has two derivatives, as we show now. One has
\begin{equation*}
v_{0_{\xi}}(\xi,\eta)=\frac14q\left(\frac{\xi}2\right),
\quad 
v_{0_{\eta}}(\xi,\eta)=-\frac14q\left(\frac{\eta}2\right),
\end{equation*}
\begin{multline}\label{v xi}
\vt_{\xi}(\xi,\eta)=-\frac14\int_{\xi}^{\eta}q\left(\frac{\eta_1-\xi}2\right)v(\xi,\eta_1)d\eta_1
+\frac14\int_0^{\xi}q\left(\frac{\xi-\xi_1}2\right)v(\xi_1,\xi)d\xi_1
\\
=-\frac12\int_0^{\frac{\eta-\xi}{2}}q(\tau)v(\xi,\xi+2\tau)d\tau+\frac12\int_0^{\frac{\xi}2}q(\tau)v(\xi-2\tau,\xi)d\tau
,
\end{multline}
\begin{equation}\label{v eta}
\vt_{\eta}(\xi,\eta)=-\frac14\int_0^{\xi}q\left(\frac{\eta-\xi_1}2\right)v(\xi_1,\eta)d\xi_1
=-\frac12\int_{\frac{\eta-\xi}2}^{\frac{\eta}2}q(\tau)v(\eta-2\tau,\eta)d\tau.
\end{equation}
Taking $v$ from the equation \eqref{v int eq} we get:
\begin{multline*}
\frac{d}{d\xi}v(\xi,\xi+2\tau)=\frac14\left(q\left(\frac{\xi}2\right)-q\left(\frac{\xi}2+\tau\right)\right)
\\
+\frac12\left(\int_0^{\frac{\xi}2}q(\sigma)v(\xi-2\sigma,\xi)d\sigma-\int_0^{\tau}q(\sigma)v(\xi,\xi+2\sigma)d\sigma\right.
\\
\left.-\int_{\tau}^{\tau+\frac{\xi}2}q(\sigma)v(\xi+2\tau-2\sigma,\xi+2\tau)d\sigma\right),
\end{multline*}
\begin{multline*}
\frac{d}{d\eta}v(\eta-2\tau,\eta)
=\frac14\left(q\left(\frac{\eta}2-\tau\right)-q\left(\frac{\eta}2\right)\right)
\\
+\frac12\left(\int_0^{\frac{\eta}2-\tau}q(\sigma)v(\eta-2\tau-2\sigma,\eta-2\tau)d\sigma-\int_0^{\tau}q(\sigma)v(\eta-2\tau,\eta-2\tau+2\sigma)d\sigma\right.
\\
\left.-\int_{\tau}^{\frac{\eta}2}q(\sigma)v(\eta-2\sigma,\eta)d\sigma\right).
\end{multline*}
Substitution to \eqref{v xi} and \eqref{v eta} gives:
\begin{multline}\label{v-xi-xi}
\vt_{\xi\xi}(\xi,\eta)=\frac{q((\eta-\xi)/2)v(\xi,\eta)+q(\xi/2)v(0,\xi)}4
\\
+\frac18\left(\int_0^{\frac{\eta-\xi}2}q(\tau)(q(\xi/2+\tau)-q(\xi/2))d\tau+\int_0^{\frac{\xi}2}q(\tau)(q(\xi/2-\tau)-q(\xi/2))d\tau\right)
\\
+\frac14\left(\int_0^{\frac{\eta-\xi}2}d\tau q(\tau)
\left(
\int_0^{\tau}d\sigma q(\sigma)v(\xi,\xi+2\sigma)-\int_0^{\frac{\xi}2}d\sigma q(\sigma)v(\xi-2\sigma,\xi)
\right.
\right.
\\
\left.
\left.
+\int_{\tau}^{\tau+\frac{\xi}2}d\sigma q(\sigma)v(\xi+2\tau-2\sigma,\xi+2\tau)\right)
\right.
\\
+\int_0^{\frac{\xi}2}d\tau q(\tau)\left(
\int_0^{\frac{\xi}2-\tau}d\sigma q(\sigma)v(\xi-2\tau-2\sigma,\xi-2\tau)
\right.
\\
\left.
\left.
-\int_0^{\tau}d\sigma q(\sigma)v(\xi-2\tau,\xi-2\tau+2\sigma)
-\int_{\tau}^{\frac{\xi}2}d\sigma q(\sigma)v(\xi-2\sigma,\xi)
\right)
\right),
\end{multline}
\begin{multline}\label{v-eta-eta}
\vt_{\eta\eta}(\xi,\eta)=\frac{q((\eta-\xi)/2)v(\xi,\eta)-q(\eta/2)v(0,\eta)}4
\\
+\frac18\int_{\frac{\eta-\xi}2}^{\frac{\eta}2}q(\tau)(q(\eta/2)-q(\eta/2-\tau))d\tau
\\
+\frac14\left(\int_{\frac{\eta-\xi}2}^{\frac{\eta}2}d\tau q(\tau)
\left(
\int_0^{\tau}d\sigma q(\sigma)v(\eta-2\tau,\eta-2\tau+2\sigma)
\right.
\right.
\\
\left.
\left.
-\int_0^{\frac{\eta}2-\tau}d\sigma q(\sigma)v(\eta-2\tau-2\sigma,\eta-2\tau)
+\int_{\tau}^{\frac{\eta}2}d\sigma q(\sigma)v(\eta-2\sigma,\eta)\right)
\right),
\end{multline}
and
\begin{equation*}
\vt_{\xi\eta}(\xi,\eta)=-\frac{q((\eta-\xi)/2)v(\xi,\eta)}4.
\end{equation*}

Returning to the function $w$ we can accordingly write it in the form $w=w_0+\wt$ with
\begin{equation*}
w_0(x,t)=-\frac12\int_{\frac{t-x}2}^{\frac{t+x}2}q(s)ds.
\end{equation*}
For the second summand $\wt$ let us find the second order derivatives
\begin{multline*}
\wt_{tt}((\eta-\xi)/2,(\eta+\xi)/2)=\vt_{\xi\xi}(\xi,\eta)+\vt_{\eta\eta}(\xi,\eta)+2\vt_{\xi\eta}(\xi,\eta)
\\
=\frac{q(\xi/2)v(0,\xi)-q(\eta/2)v(0,\eta)}4
\\
+\frac18\left(
q(\eta/2)\int_{\frac{\eta-\xi}2}^{\frac{\eta}2}q(\tau)d\tau
-q(\xi/2)\int_0^{\frac{\eta-\xi}2}q(\tau)d\tau-q(\xi/2)\int_0^{\frac{\xi}2}q(\tau)d\tau
\right.
\\
\left.
+
\int_0^{\frac{\eta-\xi}2}q(\tau)q(\xi/2+\tau)d\tau+\int_0^{\frac{\xi}2}q(\tau)q(\xi/2-\tau)d\tau
-\int_{\frac{\eta-\xi}2}^{\frac{\eta}2}q(\tau)q(\eta/2-\tau)d\tau\right)
\\
+\widehat w((\eta-\xi)/2,(\eta+\xi)/2),
\end{multline*}
where the function $\widehat w$ contains all the remaining integral terms and is continuous. Returning to the original variables we have
\begin{multline}\label{w-tt}
\wt_{tt}(x,t)
=\frac14q\left(\frac{t-x}2\right)w\left(\frac{t-x}2,\frac{t-x}2\right)
-\frac14q\left(\frac{t+x}2\right)w\left(\frac{t+x}2,\frac{t+x}2\right)
\\
+\frac18\left(
q\left(\frac{t+x}2\right)\int_{t}^{\frac{t+x}2}q(\tau)d\tau
-q\left(\frac{t-x}2\right)\int_0^{t}q(\tau)d\tau
\right.
\\
\left.
-q\left(\frac{t-x}2\right)\int_0^{\frac{t-x}2}q(\tau)d\tau
+\int_0^xq(\tau)q\left(\frac{t-x}2+\tau\right)d\tau
\right.
\\
\left.
+\int_0^{\frac{t-x}2}q(\tau)q\left(\frac{t-x}2-\tau\right)d\tau
-\int_x^{\frac{t+x}2}q(\tau)q\left(\frac{t+x}2-\tau\right)d\tau
\right)
+\widehat w(x,t).
\end{multline}
At the same time, using \eqref{v2},
\begin{multline*}
\wt_{xx}((\eta-\xi)/2,(\eta+\xi)/2)=\vt_{\xi\xi}(\xi,\eta)+\vt_{\eta\eta}(\xi,\eta)-2\vt_{\xi\eta}(\xi,\eta)
\\
=\wt_{tt}((\eta-\xi)/2,(\eta+\xi)/2)-4\vt_{\xi\eta}(\xi,\eta)
\\
=\wt_{tt}((\eta-\xi)/2,(\eta+\xi)/2)+q((\eta-\xi)/2)v(\xi,\eta),
\end{multline*}
and thus
\begin{equation*}
\wt_{xx}(x,t)=\wt_{tt}(x,t)+q(x)w(x,t).
\end{equation*}

The following lemma is a consequence of the Fubini's theorem.

\begin{Lemma}\label{lemma convolution}
If $q\in \mathcal L_{1,\rm loc}([0,\infty);\matr)$, then the integrals $\int_0^{\frac{t-x}2}q(\tau)q\left(\frac{t-x}2-\tau\right)d\tau$, $\int_x^{\frac{t+x}2}q(\tau)q\left(\frac{t+x}2-\tau\right)d\tau$ and $\int_0^xq(\tau)q\left(\frac{t-x}2+\tau\right)d\tau$ as functions of $x$ and $t$ are locally summable in both variables: for every $t>0$ they belong to $\mathcal L_1([0,t];\matr)$ as functions of $x$, and for every $x>0$ to $\mathcal L_{1,\rm loc}([x,\infty);\matr)$ as functions of $t$.
\end{Lemma}

\begin{proof}
The first integral equals $p((t-x)/2)$, where 
\begin{equation*}
p(x):=\int_0^xq(\tau)q(x-\tau)d\tau.
\end{equation*}
Take $T>0$. Then
\begin{multline*}
\int_0^T\|p(x)\|_{\matr}dx
\leqslant 
\int_0^Tdx\int_0^xd\tau \|q(\tau)\|\|q(x-\tau)\|
\\
=
\int_0^Td\tau\|q(\tau)\|\int_{\tau}^Tdx\|q(x-\tau)\|
=
\int_0^Td\tau\|q(\tau)\|\int_0^{T-\tau}d\sigma\|q(\sigma)\|
\\
\leqslant
\int_0^T\|q(\tau)\|d\tau\int_0^T\|q(\sigma)\|d\sigma
<
\infty,
\end{multline*}
which justifies the change of the order of integration. Hence $p\in \mathcal L_{1,\rm loc}([0,\infty);\matr)$, and the assertion for the first integral follows.

The second integral can be written as the difference of 
$$
\int_0^{\frac{t+x}2}q(\tau)q\left(\frac{t+x}2-\tau\right)d\tau=p\left(\frac{t+x}2\right)
$$ 
and $\int_0^xq(\tau)q\left(\frac{t+x}2-\tau\right)d\tau$. It therefore remains to consider $\int_0^xq(\tau)q\left(\frac{t+x}2-\tau\right)d\tau$ and $\int_0^xq(\tau)q\left(\frac{t-x}2+\tau\right)d\tau$. For every $t\geqslant0$ one has
\begin{multline*}
\int_0^tdx\int_0^xd\tau\|q(\tau)\|\left\|q\left(\frac{t+x}2-\tau\right)\right\|
=
\int_0^td\tau\|q(\tau)\|\int_{\tau}^tdx\left\|q\left(\frac{t+x}2-\tau\right)\right\|
\\
=2\int_0^td\tau\|q(\tau)\|\int_{\frac{t-\tau}2}^{t-\tau}d\sigma\|q(\sigma)\|
\leqslant
2\int_0^t\|q(\tau)\|d\tau\int_0^t\|q(\sigma)\|d\sigma<\infty,
\end{multline*}
\begin{multline*}
\int_0^tdx\int_0^xd\tau\|q(\tau)\|\left\|q\left(\frac{t-x}2+\tau\right)\right\|
=
\int_0^td\tau\|q(\tau)\|\int_{\tau}^tdx\left\|q\left(\frac{t-x}2+\tau\right)\right\|
\\
=2\int_0^td\tau\|q(\tau)\|\int_{\tau}^{\frac{t+\tau}2}d\sigma\|q(\sigma)\|
\leqslant
2\int_0^t\|q(\tau)\|d\tau\int_0^t\|q(\sigma)\|d\sigma<\infty.
\end{multline*}
For every $x\geqslant0$ and $T>x$ one has:
\begin{multline*}
\int_x^Tdt\int_0^xd\tau\|q(\tau)\|\left\|q\left(\frac{t+x}2-\tau\right)\right\|
=
\int_0^xd\tau\|q(\tau)\|\int_x^Tdt\left\|q\left(\frac{t+x}2-\tau\right)\right\|
\\
=2\int_0^xd\tau\|q(\tau)\|\int_{x-\tau}^{\frac{T+x}2-\tau}d\sigma\|q(\sigma)\|
\leqslant
2\int_0^x\|q(\tau)\|d\tau\int_0^{\frac{T+x}2}\|q(\sigma)\|d\sigma<\infty,
\end{multline*}
\begin{multline*}
\int_x^Tdt\int_0^xd\tau\|q(\tau)\|\left\|q\left(\frac{t-x}2+\tau\right)\right\|
=
\int_0^xd\tau\|q(\tau)\|\int_x^Tdt\left\|q\left(\frac{t-x}2+\tau\right)\right\|
\\
=2\int_0^xd\tau\|q(\tau)\|\int_{\tau}^{\frac{T-x}2+\tau}d\sigma\|q(\sigma)\|
\leqslant
2\int_0^x\|q(\tau)\|d\tau\int_0^{\frac{T+x}2}\|q(\sigma)\|d\sigma<\infty.
\end{multline*}
From this the assertion for the second and the third integrals follows, which completes the proof.
\end{proof}

We have proven the following result.

\begin{Lemma}
Let $q\in \mathcal L_{1,\rm loc}([0,\infty);\matr)$. Then
\\
1. For every $t\geqslant0$ one has $\wt_{tt}(\cdot,t)\in \mathcal L_1([0,t];\matr)$.
\\
2. For every $x\geqslant0$ one has $\wt_{tt}(x,\cdot)\in \mathcal L_{1,\rm loc}([x,\infty);\matr)$.
\\
3. For every $t\geqslant0$ the equation $\wt_{tt}(x,t)-\wt_{xx}(x,t)+q(x)w(x,t)=0$ holds for a.\,e. $x\in[0,t]$. 
\\
4. If a representative of the equivalence class $q$ is chosen, then for a.\,e. $x\geqslant0$ the equation $\wt_{tt}(x,t)-\wt_{xx}(x,t)+q(x)w(x,t)=0$ holds for a.\,e. $t\in[x,\infty)$.
\end{Lemma}

\begin{Remark}
The equation $w_{tt}(x,t)-w_{xx}(x,t)+q(x)w(x,t)=0$ can hold nowhere in the sense of an equality of partial derivatives.
\end{Remark}

\begin{Theorem}\label{theorem solution}
Let $q\in \mathcal L_{1,\rm loc}([0,\infty);\matr)$, $f\rmv\in \mathcal C^{\infty}_{\rm loc}([0,\infty);\mathbb C^n)$ and $\supp f\rmv\subset(0,\infty)$. Then the function
\begin{equation}\label{formula for u-f-hat}
u^{f\rmv}(x,t)=f\rmv(t-x)+\int_x^tw(x,s)f\rmv(t-s)ds
\end{equation}
solves the system
\begin{align}
\label{u1}& u_{tt}(x,t)-u_{xx}(x,t)+q(x)u(x,t) = 0,  && t>0,\text{ a.\,e. }x\in(0,t),\\
\label{u2}& u(x,t)=0, && x\gs t,\\
\label{u3}& u(0,t) = f\rmv(t), && t\gs 0.
\end{align}
Besides that one has 
$$
u^{f\rmv}(\cdot,t)\in \mathcal W^2_1([0,t];\mathbb C^n),\quad -u^{f\rmv}_{xx}(\cdot,t)+q(\cdot)u^{f\rmv}(\cdot,t)\in \mathcal C([0,t];\mathbb C^n)
$$ 
for every $t\geqslant0$.
\end{Theorem}

\begin{proof}
Firstly, one can differentiate twice the equality \eqref{formula for u-f-hat} in $t$ and get
\begin{equation*}
u^{f\rmv}_{tt}(x,t)=f\rmv''(t-x)+\int_x^tw(x,s)f\rmv''(t-s)ds.
\end{equation*}
Note that from this $u^{f\rmv}_{tt}(\cdot,t)\in \mathcal C([0,t];\mathbb C^n)$ for every $t\geqslant0$. As one can see, for every $x$ the functions $\wt(x,\cdot)$ and $\wt_t(x,\cdot)$ are absolutely continuous. Thus one can integrate by parts:
\begin{multline*}
u^{f\rmv}_{tt}(x,t)=f\rmv''(t-x)+\int_x^tw_0(x,s)f\rmv''(t-s)ds+\int_x^t\wt(x,s)f\rmv''(t-s)ds
\\
=f\rmv''(t-x)+\int_x^tw_0(x,s)f\rmv''(t-s)ds+\wt_{t}(x,x)f\rmv(t-x)+\int_x^t\wt_{tt}(x,s)f\rmv(t-s)ds,
\end{multline*}
since $\wt(x,x)=0$. Differentiate now \eqref{formula for u-f-hat} in $x$,
\begin{multline*}
u^{f\rmv}_{xx}(x,t)=f\rmv''(t-x)+\left(\int_x^tw_0(x,s)f\rmv(t-s)ds\right)_{xx}
\\
+\left(\int_x^t\wt(x,s)f\rmv(t-s)ds\right)_{xx}
=f\rmv''(t-x)+\left(\int_x^tw_0(x,s)f\rmv(t-s)ds\right)_{xx}
\\-\wt_x(x,x)f\rmv(t-x)+\int_x^t\wt_{xx}(x,s)f\rmv(t-s)ds,
\end{multline*}
again using $\wt(x,x)=0$. Then for every $t\geqslant0$ and a.\,e. $x\in(0,t)$
\begin{multline*}
u^{f\rmv}_{tt}(x,t)-u^{f\rmv}_{xx}(x,t)=\int_x^tw_0(x,s)f\rmv''(t-s)ds-\left(\int_x^tw_0(x,s)f\rmv(t-s)ds\right)_{xx}
\\
+(\wt_t(x,x)+\wt_x(x,x))f\rmv(t-x)+\int_x^t(\wt_{tt}(x,s)-\wt_{xx}(x,s))f\rmv(x,s)ds
\\
=\int_x^tw_0(x,s)f\rmv''(t-s)ds-\left(\int_x^tw_0(x,s)f\rmv(t-s)ds\right)_{xx}
\\
-q(x)\int_x^t\wt(x,s)f\rmv(x,s)ds,
\end{multline*}
where we used the equality $\wt_t(x,x)+\wt_x(x,x)=0$. Consider the first and the second terms separately:
\begin{multline*}
\int_x^tw_0(x,s)f\rmv''(t-s)ds
=w_0(x,x)f\rmv'(t-x)+\int_x^tw_{0_t}(x,s)f\rmv'(t-s)ds
\\
=-\frac{f\rmv'(t-x)}2\int_0^xq(\tau)d\tau-\frac14\int_x^t\left(q\left(\frac{s+x}2\right)-q\left(\frac{s-x}2\right)\right)f\rmv'(t-s)ds
\end{multline*}
and
\begin{multline*}
\left(\int_x^tw_0(x,s)f\rmv(t-s)ds\right)_{xx}
=\left(-\frac12\int_x^tds\int_{\frac{s-x}2}^{\frac{s+x}2}d\tau q(\tau)f\rmv(t-s)\right)_{xx}
\\
=\frac12\left(\left(\int_0^xq(\tau)d\tau\right)f\rmv(t-x)-\int_x^t\left(q\left(\frac{s+x}2\right)+q\left(\frac{s-x}2\right)\right)f\rmv(t-s)ds\right)_x
\\
=\frac{q(x)f\rmv(t-x)}2-\frac{f\rmv'(t-x)}2\int_0^xq(\tau)d\tau
\\
-\frac12\left(\int_x^{\frac{t+x}2}q(\sigma)f\rmv(t+x-2\sigma)d\sigma+\int_0^{\frac{t-x}2}q(\sigma)f\rmv(t-x-2\sigma)d\sigma\right)_x
\\
=q(x)f\rmv(t-x)-\frac{f\rmv'(t-x)}2\int_0^xq(\tau)d\tau
\\
-\frac12\int_x^{\frac{t+x}2}q(\sigma)f'(t+x-2\sigma)d\sigma+\frac12\int_0^{\frac{t-x}2}q(\sigma)f'(t-x-2\sigma)d\sigma
\\
=q(x)f\rmv(t-x)-\frac{f\rmv'(t-x)}2\int_0^xq(\tau)d\tau
\\-\frac14\int_x^t\left(q\left(\frac{s+x}2\right)-q\left(\frac{s-x}2\right)\right)f\rmv'(t-s)ds,
\end{multline*}
thus one has
\begin{equation*}
\int_x^tw_0(x,s)f\rmv''(t-s)ds-\left(\int_x^tw_0(x,s)f\rmv(t-s)ds\right)_{xx}=-q(x)f\rmv(t-x)
\end{equation*}
and
\begin{multline*}
u^{f\rmv}_{tt}(x,t)-u^{f\rmv}_{xx}(x,t)
\\
=-q(x)\left(f\rmv(t-x)+\int_x^tw(x,s)f\rmv(t-s)ds\right)=-q(x)u^{f\rmv}(x,t).
\end{multline*}
We see that \eqref{u1} holds. For every $t\geqslant0$ one has $u^{f\rmv}_{xx}(\cdot,t)\in \mathcal L_1([0,t];\mathbb C^n)$, which means that $u^{f\rmv}(\cdot,t)\in \mathcal H^2_1([0,t];\mathbb C^n)$ and that $-u^{f\rmv}_{xx}(x,t)+q(x)u^{f\rmv}(x,t)=-u^{f\rmv}(x,t)$ is a continuous function of $x$. Condition \eqref{u2} is obviously satisfied and \eqref{u3} holds, because $w(0,t)\equiv0$.
\end{proof}

\section{Smooth waves}\label{section 4}

In this section we show that the solution $u^{f\rmv}_{\beta}(x,t)$ to the system $\beta^T\rmv$ is at the same time a solution of the system $\alpha^T\rmv$. Since the latter is known to be unique, see \eqref{abstract representation for solution}, this means that $u^{f\rmv}_{\alpha}=u^{f\rmv}_{\beta}$.

\begin{Theorem}\label{theorem 2}
Let a locally summable Hermitian $\matr$-valued potential $q$ be such that the operator $L_0$ is positive definite. Then $u^{f\rmv}_{\beta}(x,t)$ given by \eqref{u hat f} is the solution of the system $\alpha^T\rmv$, \eqref{alpha_v1}--\eqref{alpha_v4}.
\end{Theorem}

\begin{proof}
By Theorem \ref{theorem solution} the solution $u^{f\rmv}_{\beta}$ satisfies \eqref{alpha_v1}, clearly \eqref{alpha_v3} and \eqref{alpha_v4} are satisfied as well. It remains to see that \eqref{alpha_v2} holds. This means we need to show that the derivative $\ddot u^{f\rmv}_{\beta}$ in $\mH$ exists and coincides with the partial derivative $(u^{f\rmv}_{\beta})_{tt}$. This follows from the next lemma (applied twice).
\end{proof}

\begin{Lemma}
If $f\rmv\in \mathcal C^{\infty}([0,\infty);\mathbb C^n)$ and $\supp f\rmv\subset(0,\infty)$, then $\dot u^{f\rmv}_{\beta}=(u^{f\rmv}_{\beta})_t$.
\end{Lemma}

\begin{proof}
The assertion follows from a direct estimate. Let $h\in(-1,1)\backslash\{0\}$. Since
\begin{equation*}
(u^{f\rmv}_{\beta})_t(x,t)=f\rmv'(t-x)+\int_x^tw(x,s)f\rmv'(t-s)ds=u^{f'\rmv}_{\beta}(x,t)
\end{equation*}
and $f\rmv'$ is again from $\mathcal C^{\infty}([0,\infty);\mathbb C^n)$ with $\supp f'\rmv\subset(0,\infty)$, one can differentiate any number of times, which means that $u^{f\rmv}(x,\cdot)$ is a smooth function. Using the Taylor formula 
\begin{equation*}
g(x+h)=g(x)+g'(x)h+\frac12\int_x^{x+h}(x+h-t)g''(t)dt
\end{equation*}
for a smoooth function $g$, which gives the estimate
\begin{equation*}
\left\|\frac{g(x+h)-g(x)}h-g'(x)\right\|_{\bbc}\leqslant\frac{|h|}2\max_{t\in[x,x+h]}\|g''(t)\|_{\bbc},
\end{equation*}
we obtain for $u^{f\rmv}(x,\cdot)$ the following estimates, assuming $|h|<1$:
\begin{multline*}
\int_0^{\infty}\left\|
\frac{u^{f\rmv}(x,t+h)-u^{f\rmv}(x,t)}h-u^{f\rmv}_t(x,t)
\right\|^2_{\mathbb C^n}dx
\\
\leqslant\frac{h^2}2
\int_0^{\infty}\max_{t_1\in[t,t+h]}\left\|
u^{f\rmv}_{tt}(x,t_1)
\right\|_{\mathbb C^n}^2dx
\\
=\frac{h^2}2\int_0^{t+h}\max_{t_1\in[t,t+h]}\left\|
f\rmv''(t_1-x)+\int_x^{t_1}w(x,s)f\rmv''(t_1-s)ds
\right\|_{\mathbb C^n}^2dx
\\
\leqslant\frac{h^2}2
\max_{t_1\in[0,t+1]}\|f\rmv''(t_1)\|_{\mathbb C^n}^2
(t+1)\Big[1+\max_{t_1\in[0,t+1],x_1\in[0,t_1]}\|w(x_1,t_1)\|_{\matr}^2(t+1)\Big]^2.
\end{multline*}
This means that $\frac{u^{f\rmv}(\cdot,t+h)-u^{f\rmv}(\cdot,t)}h\to u^{f\rmv}_t(\cdot,t)$ as $h\to0$ in the space $\mH=\mathcal L_2([0,\infty);\mathbb C^n)$.
\end{proof}

The next corollary follows immediately.

\begin{Corollary}
The function $u^f(x,t)=u^{f\rmv}_{\beta}(x,t)$, where $f=\Lambda^T f\rmv$, is the solution of the system $\alpha^T$, \eqref{alpha1}--\eqref{alpha4}. The closure of the operator $W^T$ is an isomorphism from $\mF^T$ to $\mU^T$ and $\mU^T=\mathcal L_2(\otbbc)$.
\end{Corollary}

\section{The control operator and the Sobolev norm}\label{section 5}

The operator $W_{\rm v}^T:f\rmv\mapsto u^{f\rmv}_{\beta}(\cdot,t)$ defined by \eqref{u hat f} in $\mathcal L_2([0,T];\mathbb C^n)$ can be at the same time considered as an operator acting in the Sobolev space $\mathcal H^2([0,T];\mathbb C^n)$. One can see that it is a bounded operator in the norm of this space. We show that it is also boundedly invertible in this norm.

\begin{Theorem}\label{theorem isomorphism}
If $q\in \mathcal L_2([0,T];\matr)$, then the operator $W_{\rm v}^T$ given by \eqref{relation btw W} and restricted to the space $\mathcal H^2([0,T];\mathbb C^n)$ is an isomorphism in this space.
\end{Theorem}

\begin{proof}
Let $Y^T:f\rmv(\cdot)\mapsto f\rmv(T-\cdot)$ be the reflection operator in $\mathcal H^2([0,T];\mathbb C^n)$.
The property of being an isomorphism in $\mathcal H^2([0,T];\mathbb C^n)$ holds simultaneously for $W_{\rm v}^T$ and for the operator $\tilde W^T:=W_{\rm v}^TY^T=I+A^T$ where the operator $A^T$ acts in $\mathcal H^2([0,T];\mathbb C^n)$ by the rule
\begin{equation*}
(A^Tf\rmv)(x)=(\mathcal A f\rmv)(x):
=
\int_x^T
\left(
\wt(x,s)-\int_{\frac{s-x}2}^{\frac{s+x}2}\frac{q(\tau)}2d\tau
\right)
f\rmv(s)ds.
\end{equation*}
One has:
\begin{multline*}
(A^Tf\rmv)'(x)
=
\left(
\int_0^x\frac{q(\tau)}2d\tau 
\right)
f\rmv(x)
\\
+
\int_x^T
\left(
\wt_x(x,s)-\frac{q\left(\frac{s+x}2\right)+q\left(\frac{s-x}2\right)}4
\right)
f\rmv(s)ds,
\end{multline*}
\begin{multline*}
(A^Tf\rmv)''(x)
=
(q(x)-\wt_x(x,x))f\rmv(x)+
\left(
\int_0^x\frac{q(\tau)}2d\tau
\right)
f\rmv'(x)
\\
+
\left(
\frac{q\left(\frac{T-x}2\right)-q\left(\frac{T+x}2\right)}4 
\right)
f\rmv(T)
\\
+
\int_x^T\wt_{xx}(x,s)f\rmv(s)ds
-
\int_x^T\left(\frac{q\left(\frac{s+x}2\right)-q\left(\frac{s-x}2\right)}4\right)f\rmv'(s)ds.
\end{multline*}
From these equalities it is immedately clear that $A^T$ is a bounded operator in $\mathcal H^2([0,T];\mathbb C^n)$.

The expression $\mathcal A$ defines several operators in different spaces. In $\mathcal L_2([0,T];\mathbb C^n)$ it defines the operator $A^T_{\mathcal L}$ which is a Volterra integral operator, thus $(I+A^T_{\mathcal L})^{-1}$ exists, is bounded and the Neumann's series converges in $\mathfrak B(\mathcal L_2([0,T];\mathbb C^n))$:
\begin{equation}\label{Neumann series}
(I+A^T_{\mathcal L})^{-1}=\sum_{n=0}^{\infty}(-1)^n(A^T_{\mathcal L})^n
=
I-A^T_{\mathcal L}+(A^T_{\mathcal L})^2-(A^T_{\mathcal L})^3(I+A^T_{\mathcal L})^{-1}.
\end{equation}
Let us show that the expression $\mathcal A$ defines bounded operators
\begin{enumerate}
\item from $\mathcal L_2([0,T];\mathbb C^n)$ to $\mathcal C([0,T];\mathbb C^n)$,
\item from $\mathcal C([0,T];\mathbb C^n)$ to $\mathcal C^1([0,T];\mathbb C^n)$,
\item form $\mathcal C^1([0,T];\mathbb C^n)$ to $\mathcal H^2([0,T];\mathbb C^n)$.
\end{enumerate}
These operators and $A^T$ can be regarded as restrictions of $A^T_{\mathcal L}$.
Denote 
\begin{align*}
a_1&:=1/2\|q\|_{\mathcal L_1([0,T];\matr)},
\\
a_2&:=\|q\|_{\mathcal L_2([0,T];\matr)},
\\
b_1&:=\|\wt\|_{\mathcal C(\{(x,t)\,|\,t\in[0,T],x\in[0,t]\};\matr)},
\\
b_2&:=\|\wt_x\|_{\mathcal C(\{(x,t)\,|\,t\in[0,T],x\in[0,t]\};\matr)},
\\
b_3&:=\int_0^Tdx\left(\int_x^T\|\wt_{xx}(x,t)\|_{\matr}dt\right)^2.
\end{align*}
We need to show that $b_3$ is finite. Since $\wt_{xx}(x,t)=\wt_{tt}(x,t)+q(x)w(x,t)$, $w$ is continuous and $q\in \mathcal L_2([0,T];\matr)$, we need to prove that 
$$
\int_0^Tdx\left(\int_x^T\|\wt_{tt}(x,t)\|_{\matr}dt\right)^2<\infty.
$$
It is enough to show that $\int_x^T\|\wt_{tt}(x,t)\|_{\matr}dt$ is a bounded function of $x\in[0,T]$. Looking at \eqref{w-tt} and \eqref{v-xi-xi}--\eqref{v-eta-eta} to estimate $\hat w$, with
$$
b_4:=\|w\|_{\mathcal C(\{(x,t)\,|\,t\in[0,T],x\in[0,t]\};\matr)},
$$
for every $x\in[0,T]$ one can write, using the estimates from the proof of Lemma \ref{lemma convolution}:
\begin{equation*}
\int_x^T\|\wt_{tt}(x,t)\|_{\matr}dt\leqslant2a_1b_4+\frac{3a_1^2}4+\frac{4a_1^2+2a_1^2}8+\frac{9a_1^2b_4}4.
\end{equation*}
From this one has $b_3<\infty$. Then under corresponding assumptions about $f\rmv$ the following estimates hold:
\begin{enumerate}[label=\roman*.]
\item $\|A^Tf\rmv\|_{\mathcal C}\leqslant(a_1+b_1)\|f\rmv\|_{\mathcal L_1}\leqslant(a_1+b_1)\sqrt T\|f\rmv\|_{\mathcal L_2}$,
\item
$\|(A^Tf\rmv)'\|_{\mathcal C}\leqslant a_1\|f\rmv\|_{\mathcal C}+b_2 T\|f\rmv\|_{\mathcal C}+2a_1\|f\rmv\|_{\mathcal C}=(3a_1+b_2 T)\|f\rmv\|_{\mathcal C}$,
\\ $\|A^Tf\rmv\|_{\mathcal C}\leqslant(a_1+b_1)\|f\rmv\|_{\mathcal L_1}\leqslant(a_1+b_1)T\|f\rmv\|_{\mathcal C}$,
\item
$\|(A^Tf\rmv)''\|_{\mathcal L_2}\leqslant(a_2+b_2\sqrt T)\|f\rmv\|_{\mathcal C}+a_1\sqrt T\|f\rmv'\|_{\mathcal C}+2a_1\|f\rmv\|_{\mathcal C}+\sqrt b_3\|f\rmv\|_{\mathcal C}+2a_1\|f\rmv'\|_{\mathcal C}\leqslant(4a_1+a_2+(a_1+b_2)\sqrt T+\sqrt{b_3})
\|f\rmv\|_{\mathcal C^1}$, 
\\
$\|A^Tf\rmv\|_{\mathcal L_2}\leqslant\sqrt T\|A^Tf\rmv\|_{\mathcal C}\leqslant(a_1+b_1)T^{\frac32}\|f\rmv\|_{\mathcal C}\leqslant(a_1+b_1)T^{\frac32}\|f\rmv\|_{\mathcal C^1}$.
\end{enumerate}
This means that the claims 1.--3. hold true. Therefore $(A^T_{\mathcal L})^3$ can be viewed as a bounded operator from $\mathcal L_2([0,T];\mathbb C^n)$ to $\mathcal H^2([0,T];\mathbb C^n)$. Since the embedding operator $J_{\mathcal H^2}$ of $\mathcal H^2([0,T];\mathbb C^n)$ into $\mathcal L_2([0,T];\mathbb C^n)$ is bounded, from \eqref{Neumann series} one can see that the restriction $(I+A^T_{\mathcal L})^{-1}\upharpoonright\mathcal H^2([0,T];\mathbb C^n)$ is a bounded operator in $\mathcal H^2([0,T];\mathbb C^n)$:
\begin{equation*}
(I+A^T_{\mathcal L})^{-1}\upharpoonright\mathcal H^2([0,T];\mathbb C^n)=
I-A^T+(A^T)^2-(A^T_{\mathcal L})^3(I+A^T_{\mathcal L})^{-1}J_{\mathcal H^2}.
\end{equation*}
Together with the fact that $I+A^T=(I+A^T_{\mathcal L})\upharpoonright\mathcal H^2([0,T];\mathbb C^n)$ this implies the equalities
\begin{equation*}
(I+A^T_{\mathcal L})^{-1}\upharpoonright\mathcal H^2([0,T];\mathbb C^n)
(I+A^T)
=I_{\mathcal H^2}=
(I+A^T)
(I+A^T_{\mathcal L})^{-1}\upharpoonright\mathcal H^2([0,T];\mathbb C^n),
\end{equation*}
which means that
\begin{equation*}
(I+A^T_{\mathcal L})^{-1}\upharpoonright\mathcal H^2([0,T];\mathbb C^n)=(I+A^T)^{-1}.
\end{equation*}
It follows that $\tilde W^T=I+A^T$ and hence $W^T\rmv$ are isomorphisms in $\mathcal H^2([0,T];\mathbb C^n)$. This completes the proof.
\end{proof}

\section*{Acknowledgements}
The author expresses his deep gratitude to Prof. M.\,I.\,Belishev for his attention to this work and for many fruitful discussions of the subject.

\bigskip
\noindent{\bf Keywords:} Initial-boundary value problem, telegraph equation, matrix Schr\"odinger operator, BC-method, control operator, Goursat problem.

\smallskip
\noindent{\bf MSC:}
35L20, 
47B93, 
47B25. 
\end{document}